\documentclass[11pt]{amsart}

\usepackage{amsmath, amssymb,  mathabx, amsfonts,enumerate}
\usepackage[all]{xy}
\usepackage{amscd,stmaryrd}
\usepackage{comment}
\usepackage{mathtools}
\usepackage{tikz} 
\usepackage{tikz-cd} 
\tikzcdset{diagrams={nodes={inner sep=7.5pt}}}

\usepackage{url}
\usepackage{hyperref}

\newtheorem{theorem}{Theorem}[section]
\newtheorem{lemma}[theorem]{Lemma}
\newtheorem{corollary}[theorem]{Corollary}

\newtheorem{theoremletter}{Theorem}

 \theoremstyle{definition}
 \newtheorem{definition}[theorem]{Definition}

 \newtheorem{example}[theorem]{Example}

  \newtheorem*{example*}{Example}

\numberwithin{equation}{section}
 
\newcommand {\N}{\mathbb{N}} 
\newcommand {\Z}{\mathbb{Z}} 
\newcommand {\R}{\mathbb{R}}

\newcommand{\FF}{\mathcal{F}}

\newcommand{\LL}{\mathcal{L}}

\newcommand{\PP}{\mathcal{P}}

\DeclareMathOperator{\mdim}{mdim}
\DeclareMathOperator{\mDim}{mDim}
\DeclareMathOperator{\ent}{ent}
\DeclareMathOperator{\Ent}{Ent}
\DeclareMathOperator{\Fix}{Fix}

\DeclareMathOperator{\Ker}{Ker}

\DeclareMathOperator{\im}{Im}

\begin{document}
\title[Garden of Eden theorem  for non-uniform cellular automata]{On the Garden of Eden theorem  for non-uniform cellular automata} 
\author[Xuan Kien Phung]{Xuan Kien Phung}
\address{Département d'informatique et de recherche opérationnelle,  Université de Montréal, Montréal, Québec, H3T 1J4, Canada.}
\email{phungxuankien1@gmail.com}   

\begin{abstract}
We establish several extensions of the well-known Garden of Eden theorem for non-uniform cellular automata over the full shifts and over amenable group universes. In particular, our results describe quantitatively the relations  between the partial pre-injectivity and the size of the image of a non-uniform cellular automata. A strengthened surjunctivity  result is also obtained for multi-dimensional cellular automata over strongly irreducible subshifts of finite type. 
\end{abstract}
\maketitle
  
\setcounter{tocdepth}{1}

 \section{Introduction}
The famous Garden of Eden theorem for cellular automata (CA) characterizes amenable groups and states the equivalence between the surjectivity (a global property) and the pre-injectivity (a local property) \cite{myhill}, \cite{moore},  \cite{ceccherini}, \cite{bartholdi-kielak}. Extensions of the Garden of Eden theorem are also obtained for endomorphisms of symbolic varieties to deal with more general alphabet structures \cite{cscp-alg-goe}, \cite{phung-post-surjective}. 
Other recent developments of the Garden of Eden theorems are related to the surjunctivity conjecture of Gottschalk \cite{gottschalk} for CA and the Kaplansky conjecture \cite{kap} for group rings (cf.  \cite{phung-2020}, \cite{phung-geometric}, \cite{phung-weakly}, \cite{phung-twisted}). When the uniformity in the definition of CA breaks down, the cells in the universe of a CA can follow different transition rules. In such situations, either intentionally or  by undesirable perturbations,  many interesting and fundamental properties of CA are still preserved such as the surjunctivity over amenable or residually finite group universes, the closed image property, the shadowing property, and several decidable properties (cf. \cite{phung-israel},  \cite{phung-tcs}, \cite{phung-shadowing}, \cite{phung-decidable-nuca}). 
However, the Garden of Eden is far from being valid for non-uniform CA (NUCA) as illustrated by the following example (see \cite[Example 14.3]{phung-tcs}) of a simple injective one-dimensional NUCA which is not surjective. 

\begin{example}
\label{ex:1} 
Let $A=\{0,1\}$ and consider the NUCA $\tau \colon A^\Z \to A^\Z$ defined for every $x \in A^\Z$  by $\tau(x)(n)=x(n+1)$ if $n \leq -1$, $\tau(x)(n)=x(n)$ if $n=0$, and $\tau(x)(n)=x(n-1)$ if $n \geq 1$. Then $\tau$ is injective (thus pre-injective) since 
$\tau(x)=y$ implies that 
$x(n)=y(n-1)$ for $n \leq 0$ and $ x(n)=y(n+1)$ for $n \geq 0$. 
On the other hand, $\tau$ is not surjective since we can check that  
\[\im \tau = \{y \in A^\Z\colon y(-1)=y(0)=y(1)\}.
\]
\end{example}
\par 
Nevertheless, we show that the dynamics of NUCA are not completely chaotic and the above  example illustrates the worst-case behavior of pre-injective NUCA. More precisely, the principal goal of our paper is to establish several optimal extensions of the Garden of Eden over amenable group universes (Theorem~\ref{t:GOE-nuca-c-v1-intro}, Theorem~\ref{t:GOE-nuca-c-1-intro},  Theorem~\ref{t:characterization-linera-GOE-nuca-mdim-pre-intro}) which  describe quantitatively the exact relations between the partial pre-injectivity and the mean entropy or mean dimension of the image of non-uniform cellular automata. Moreover, we also obtain a refined surjunctivity result for multi-dimensional CA over strongly irreducible subshifts of finite type (Theorem~\ref{t:finite-z-d-almost-general-closed-subset-dense-periodic-intro}).   
 \par 
To state the main results, let us recall basic notions of symbolic dynamics. 
Given a discrete set $A$ and a group $G$, a \emph{configuration} $c \in A^G$ is a map $c \colon G \to A$ and  $x,y  \in A^G$ are \emph{asymptotic} if  $x\vert_{G \setminus E}=y\vert_{G \setminus E}$ for some finite subset $E \subset G$.  
The \emph{Bernoulli shift} action $G \times A^G \to A^G$ is defined by $(g,x) \mapsto g x$, 
where $(gx)(h) =  x(g^{-1}h)$ for  $g,h \in G$,  $x \in A^G$. 
We equip the \emph{full shift} $A^G$ with the \emph{prodiscrete topology}. 
Following an idea of von Neumann and Ulam  \cite{neumann}, a CA over the group $G$ (the \emph{universe}) and the set $A$ (the \emph{alphabet})  is  a self-map $A^G \righttoleftarrow$ which is $G$-equivariant and uniformly continuous (cf.~\cite{hedlund-csc}, \cite{hedlund}).  One regards elements $g \in G$ as the cells of the universe. 
When different cells evolve according to different local transition maps, we obtain  \emph{non-uniform CA} (NUCA) (\cite[Definition~1.1]{phung-tcs}, \cite{Den-12a}, \cite{Den-12b}):   

\begin{definition}
\label{d:most-general-def-asyn-ca}
Let $M$ be a subset of a group $G$  and let  $A$ be a set. Let $S = A^{A^M}$ be the set of all maps $A^M \to A$. Given $s \in S^G$, the NUCA $\sigma_s \colon A^G \to A^G$ is defined  by $
\sigma_s(x)(g)=  
    s(g)((g^{-1}x)  
	\vert_M)$ 
 for $x \in A^G$, $g \in G$. 
 \end{definition} 
\par 
The set $M$ is called a \emph{memory} and $s \in S^G$  the \textit{configuration of local transition maps} of $\sigma_s$.  Every CA is thus a NUCA with finite memory and constant configuration of local transition maps. 
The NUCA $\sigma_s$ is \emph{pre-injective} if $\sigma_s(x) = \sigma_s(y)$ implies $x= y$ whenever $x, y \in A^G$ are asymptotic.

\subsection{Main results} 
Our first result is a version of the Garden of Eden theorem for the class $\mathrm{NUCA}_c(G,A)$ consisting of NUCA with finite memory and with asymptotically constant configuration of local transition maps over the universe $G$ and the alphabet $A$. It states that the image of such an NUCA contains a nonempty open subset if and only if the NUCA is pre-injective when restricted to some nonempty open subset of the full shift.  
\par 
\begin{theoremletter}
\label{t:GOE-nuca-c-v1-intro} 
Let $G$ be an amenable group and let $A$ be a finite alphabet. Then for every $\tau \in \mathrm{NUCA}_c(G,A)$, the following are equivalent: 
\begin{enumerate}[\rm (i)] 
\item  $\tau\vert_U$ is pre-injective for some nonempty open subset $U\subset A^G$. 
    \item  $\im \tau$ contains a nonempty open subset of $A^G$. 
\end{enumerate}
\end{theoremletter}
\par 
The next example shows that Theorem~\ref{t:GOE-nuca-c-v1-intro} is optimal in the sense that both the implications fail  under the general assumption that $\tau$ is an arbitrary NUCA with finite memory even when $\tau$ is injective or surjective. 
\begin{example}
 \label{ex:2} 
 Let $A=\{0,1\}$ and consider the NUCA $\tau \colon A^{\Z^2} \to A^{\Z^2}$ defined for every $x \in A^{\Z^2}$ by $\tau(x)(m,n)=x(m,n+1)$ if $n \leq m-1$, $\tau(x)(m,m)=x(m,m)$, and $\tau(x)(m,n)=x(m,n-1)$ if $n \geq m+1$. Then as in Example~\ref{ex:1}, we find that $\tau$ is injective but   
\[\im \tau = \{y \in A^{\Z^2} \colon y(m,m-1)=y(m,m)=y(m,m+1) \text{ for all }m \in \Z \}
\]
which implies that $\im \tau$ cannot contain any nonempty open subset of $A^{\Z^2}$. 
\par 
Similarly, consider the NUCA $\sigma \colon A^{\Z^2} \to A^{\Z^2}$ given by $\sigma(x)(m,n)=x(m,n-1)$ if $n \leq m-1$, $\sigma(x)(m,m)=x(m,m-1)+x(m,m)+x(m,m+1)$ (mod 2), and $\sigma(x)(m,n)=x(m,n+1)$ if $n \geq m+1$. It is not hard to see that $\sigma$ is surjective. However, for $m \in \Z$, we have $\sigma(z_{m})=\sigma(0^{\Z^2})=0^{\Z^2}$ where $z_m(m,n)=1$ if $n\in \{m, m-1\}$ and    $z_m(m,n)=0$ otherwise. In particular,  $\tau\vert_U$ cannot be pre-injective for any nonempty open subset $U \subset A^{\Z^2}$. 
\end{example}
\par 
To deal with arbitrary NUCA, we establish  the following extension of the Myhill property for NUCA which generalizes also the implication 
(i)$\implies$(ii) in Theorem~\ref{t:GOE-nuca-c-v1-intro}. 
In essence, the size of the image of an NUCA as measured by the upper natural $\FF$-entropy $\overline{\ent}_\FF$  (cf.~Section~\ref{s:banach-entropy}) increases linearly with the size, expressed as certain upper natural $\FF$-density $\overline{d}_\FF$ (cf. Section~\ref{s:banach-densities}),  of the subset on which the restriction of the NUCA is pre-injective. 
\par 
\begin{theoremletter}
\label{t:GOE-nuca-c-1-intro}
   Let $A$ be a finite alphabet and let $S$ be a proper subset of  an amenable group $G$.  Fix $p \in A^S$ and $U=\{p\} \times A^{G \setminus S}$. 
    Suppose that $\tau \colon A^G \to A^G$ is an NUCA with finite memory such that $\tau\vert_{U}$ is pre-injective. Then $\overline{\ent}_\FF (\im \tau\vert_{U} ) \geq (1 - \overline{d}_\FF(S)) \log |A|$ for every F\o lner net $\FF$ of $G$. 
\end{theoremletter}
\par 
We deduce the classical Myhill property for CA from Theorem~\ref{t:GOE-nuca-c-1-intro} as follows.  If $S=\varnothing$ then $\overline{d}_\FF(S)=0$ and $\tau$ is pre-injective. If in addition  $\tau$ is a CA so that $\im \tau$ is a closed subshift, then by Theorem~\ref{t:GOE-nuca-c-1-intro}, $\overline{\ent}_\FF(\im \tau) =\log|A|$ is maximal thus $\im \tau = A^G$ by \cite[Proposition~4.2]{csc-myhill-strongly-irreducible}, i.e., $\tau$ is surjective. 
\par 
The next example  implies that the converse of Theorem~\ref{t:GOE-nuca-c-1-intro} does not hold. 
\begin{example}
    \label{ex:3}
    Let $\tau \colon \{0,1,2\}^\Z \to \{0,1,2\}^\Z$ be an NUCA given by $\tau(x)(n)=x(n)$ if $x(n)\neq 2$ and $\tau(x)(n)=0$ otherwise. Then $\im \tau = \{0,1\}^\Z$  and $\overline{\ent}_\FF (\im \tau )= \log 2 >0$. However, $\tau$ is not pre-injective when restricted to any subset of the form $\{p\}\times \{0,1,2\}^{\Z \setminus S}$ where $S \subsetneq \Z$ and $p \in \{0,1,2\}^S$. 
\end{example} 
\par 
When the alphabet $A$ is a vector space, the full shift $A^G$ is naturally a vector space with component-wise operations and an NUCA $\tau \colon A^G \to A^G$ is \emph{linear} if it is a linear map, or equivalently,  if its local transition maps are linear. 
 Such linear NUCA with finite memory are interesting dynamical objects as  they satisfy the shadowing property \cite{phung-shadowing}, \cite{phung-dual-linear-nuca}.  
\par 
Our next result is a quantitative version of the Garden of Eden theorem for linear NUCA which  extends  \cite{csc-goe-linear} for linear CA and relates the partial pre-injectivity of a linear NUCA as measured by  $\overline{d}_\FF$ or  the upper Banach density $\overline{D}_G$ (Section~\ref{s:banach-densities}) with the upper $\FF$-natural and lower Banach mean dimensions $\overline{\mdim}_\FF$,  $\underline{\mDim}$ (Section~\ref{s:banach-mean-dim})  of the image of the NUCA.  

\par 
\begin{theoremletter}
\label{t:characterization-linera-GOE-nuca-mdim-pre-intro}
Let $G$ be an infinite countable amenable group and let $A$ be a finite-dimensional  vector space alphabet. Then for all linear NUCA  $\tau \colon A^G \to A^G$ with finite memory, the following hold:   
\begin{enumerate} [\rm (i)]
\item If $\tau\vert_U$ is pre-injective where $U=\{0\}^S \times A^{G \setminus S}$ for some $S \subset G$,  then  $\overline{\mdim}_\FF (\im \tau ) \geq (1-\overline{d}_\FF(S))\dim A$ for every F\o lner sequence $\FF$ of $G$. 
    \item 
    If $\dim A=1$ and $\underline{\mDim} (\im \tau ) > 1-d$, then there exists $S\subset G$ with $\overline{D}_G(S)\leq d$  such that  $\tau\vert_U$ is pre-injective where $U=\{0\}^{S} \times A^{G \setminus S}$. 
\end{enumerate}
\end{theoremletter}
\par 
In particular, for linear NUCA with finite memory, Theorem~\ref{t:GOE-nuca-c-1-intro}  remains valid by replacing the mean entropy with $\overline{\mdim}_\FF$.
\par 
Fix a group $G$ and a finite  alphabet $A$, a \emph{subshift} is any $G$-invariant subset of $A^G$. We say that a (closed) subshift $X$ is of \emph{finite type} if for some finite subsets $D \subset G$ (called the \emph{window}) and $\PP \subset A^D$, the subshift $X$ consists of all $x \in A^G$ such that  $(gx)\vert_D \in \PP$ for all $g \in G$. A subshift $X \subset A^G$ is \emph{strongly irreducible} (see~\cite{burton}, \cite{ward}, \cite{lightwood}, \cite{fiorenzi}) if it is \emph{$\Delta$-irreducible} for some finite subset $\Delta \subset G$ which means that for all  $x, y\in X$ and finite subsets $S, T\subset G$ with $S\Delta \cap T=\varnothing$, we can find $z \in X$ such that  $z\vert_S=x\vert_S$ and $z\vert_T=y\vert_T$.  
\par 
For multi-dimensional CA over strongly irreducible subshifts of finite type, we obtain the following generalization of the surjunctivity property which extends notably \cite[Theorem~1.1]{csc-myhill-strongly-irreducible} since we only assume that $\tau$ is injective when restricted to a nonempty open subset of the subshift. 
\par 
\begin{theoremletter}
\label{t:finite-z-d-almost-general-closed-subset-dense-periodic-intro}
 Let $G=\Z^d$, $d \geq 1$, and let $A$ be a finite alphabet. Let $X \subset A^G$ be 
    strongly irreducible closed subshift of finite type which contains a periodic configuration.  
    Suppose that $\tau \colon X \to X$ is a  CA such that $\tau\vert_{U}$ is injective where $U \subset X$ is a nonempty open subset. Then $\tau$ is surjective and pre-injective. 
\end{theoremletter}
\par 
The above result belongs to the rich literature of the Gottschalk surjunctivity conjecture which says that every injective cellular automata with finite alphabet over the full shift must be injective (cf. \cite{gottschalk}). While the conjecture remains open, it holds for the wide class of sofic groups by a result of Gromov (cf.~\cite{gromov-esav}). Several extensions were obtained for sofic group universes but more general alphabets such as finite-dimensional vector spaces (cf.~\cite{csc-sofic-linear}) and algebraic varieties (cf.~\cite{phung-geometric}). 

\subsection{Organization of the paper} 

In Section~\ref{s:pre}, we collect preliminary definitions and basic results concerning F\o lner nets and quasi-tilings of amenable groups as well as various notions of natural and Banach mean entropies and mean dimensions of subsets of the full shift. Then in Section~\ref{s:almost-injective}, we discuss the almost injectivity property  and present the proof of Theorem~\ref{t:finite-z-d-almost-general-closed-subset-dense-periodic-intro} (see Theorem~\ref{t:finite-z-d-almost-general-closed-subset-dense-periodic}).   Section~\ref{s:myhill} contains   the proofs of Theorem~\ref{t:GOE-nuca-c-v1-intro} (see Theorem~\ref{t:GOE-nuca-c-v1}) and Theorem~\ref{t:GOE-nuca-c-1-intro} (see Theorem~\ref{t:GOE-nuca-c-1}). 
Finally, the proof of the Garden of Eden theorem for linear NUCA (Theorem~\ref{t:characterization-linera-GOE-nuca-mdim-pre-intro}) is given in Section~\ref{s:GOE-linear-NUCA} (see Theorem~\ref{t:characterization-linera-GOE-nuca-mdim-pre}). In particular, we  obtain a stronger version (see Theorem~\ref{t:linera-GOE-nuca-mdim-pre}) of  Theorem~\ref{t:characterization-linera-GOE-nuca-mdim-pre-intro}.(i) when the configuration of local transition maps of $\tau$ differs from a constant configuration only outside a subset of zero density.

\section{Preliminaries} 
\label{s:pre}

\subsection{Amenable groups and F\o lner nets}

Amenable groups were introduced by von Neumann in \cite{neumann-amenable}. 
A group $G$ is \emph{amenable} if it admits a \emph{F\o lner net}, i.e., a family $(F_i)_{i \in I}$ over a directed set $I$ consisting of nonempty finite subsets of $G$ such that (cf. \cite{folner})
\begin{equation}
\label{e:folner-s}
\lim_{i \in I} \frac{\vert F_i \setminus  F_i g\vert}{\vert F_i \vert} = 0 \text{  for all } g \in G. 
\end{equation}
\par 
Examples of amenable groups include finitely generated groups of subexponential growth and  solvable groups. However, all groups containing a subgroup isomorphic to a free group of rank 2 are non-amenable. See e.g. \cite{stan-amenable} for some more details.
\subsection{Boundary and interior sets} 
Given any subsets $F, M$ of a group $G$, we define    the \emph{$M$-boundary} of $F$ by $\partial_{M} F= FM \setminus F^{-M}$ where $F^{-M} = \{g \in F\colon gM \subset F\}$ is called the \emph{$M$-interior} subset of $F$.

\subsection{Natural and Banach densities}
\label{s:banach-densities}

\begin{definition}
    Let $G$ be an amenable group with a F\o lner net $\FF=(F_i)_{i \in I}$. For every subset $S \subset G$, let us define respectively 
  \[
  \overline{d}_\FF(S) = \limsup_{i \in I } \frac{|S \cap F_i|}{|F_i|}, \quad \underline{d}_\FF(S) = \liminf_{ i \in I} \frac{|S \cap F_i|}{|F_i|}
  \]
  as the \emph{upper natural $\FF$-density} and the \emph{lower natural $\FF$-density} of $S$ in $G$. 
  \par 
  Similarly, we define respectively 
  \[
  \overline{D}_G(S) = \liminf_{i \in I} \sup_{g \in G} \frac{|S \cap F_ig|}{|F_i|}, \quad \underline{D}_G(S) = \limsup_{i \in I} \inf_{g \in G} \frac{|S \cap F_ig|}{|F_i|}
   \] 
  as the  \emph{upper Banach density}  and the  \emph{lower Banach density} of $S$ in $G$. 
\end{definition}
 
\par 
On contrary to upper and lower natural $\FF$-densities, the following standard lemma (\cite[Lemma~2.9]{tiling-amenable}) shows that the Banach densities do not depend on the choice of the F\o lner net. 
\begin{lemma}
     Let $G$ be an amenable group with a F\o lner net $\FF=(F_i)_{i \in I}$. For every subset $S \subset G$, we have: 
\[
 \overline{D}_G(S) = \lim_{i \in I} \sup_{g \in G} \frac{|S \cap F_ig|}{|F_i|} = \inf_{F \subset G, |F|<\infty}  \sup_{g \in G} \frac{|S \cap Fg|}{|F|},
 \]
\[
 \underline{D}_G(S) = \lim_{i \in I } \sup_{g \in G} \frac{|S \cap F_ig|}{|F_i|} = \sup_{F \subset G, |F|<\infty}  \inf_{g \in G} \frac{|S \cap Fg|}{|F|}. 
 \]
\end{lemma}
\par 
We also have the following basic  properties.  

\begin{lemma}
\label{l:chain-densities}
 Let $S, T $ be subsets of a countable amenable group $G$ with a F\o lner net $\FF=(F_i)_{i \in I}$. Then 
 \[
 0\leq \underline{D}_G(S) \leq \underline{d}_\FF(S)\leq \overline{d}_\FF(S) \leq \overline{D}_G(S)\leq 1.
 \]
 Moreover, we have: 
  \begin{enumerate}[\rm (i)]
        \item
    $\overline{d}_\FF(S\cup T) \leq \overline{d}_\FF(S)+ \overline{d}_\FF(T)$ and $\overline{D}_G(S\cup T) \leq \overline{D}_G(S)+ \overline{D}_G(T)$
    \item $\overline{d}_\FF(S) + \underline{d}_\FF(G\setminus S)=1$ and $\overline{D}_G(S) + \underline{D}_G(G\setminus S)=1$
        \end{enumerate}
\end{lemma}

\begin{example}
    Let us fix $G= \Z$ and a subset $S \subset G$. Then:
    \begin{enumerate}[\rm 1.] 
        \item $\underline{D}_G(S)=\overline{D}_G(S)=0$ if $S$ is finite or $S=\{n^k\colon n \in \Z\}$ with $k\geq 2$.
        \item $\underline{D}_G(S)=\overline{D}_G(S)=\frac{1}{a}$ if  $S=\{an+b\colon n \in \Z\}$ where $a,b \in \N$ and $a >0$.   
    \end{enumerate}
\end{example}

\begin{example}
    Let $G= \Z^2$ and let  $S \subset G$. Then:
    \begin{enumerate}[\rm 1.] 
        \item $\underline{D}_G(S)=\overline{D}_G(S)=0$ if $S$ is contained in a union of   finitely many lines, that is,   $S \subset \bigcup_{i=1}^n L_n$ where for $1\leq i \leq n$ and for some $a_i,b_i,c_i \in \R$, 
        \[
        L_i=\{ (x,y) \in \Z^2 \colon a_ix+b_iy+c_i=0\}.
        \]
        \item $\underline{D}_G(S)=0$ and $\underline{D}_G(S)=1$  if $S= \Z^2 \cap H$ for any half plane $H \subset \R^2$.   
    \end{enumerate}
\end{example}
\par

\subsection{Natural and Banach mean entropies} 
\label{s:banach-entropy}

\begin{definition}
     Let $G$ be an amenable group with a F\o lner net $\FF=(F_i)_{i \in I}$ and let $A$ be a finite alphabet. We define  the \emph{upper natural mean $\FF$-entropy}, resp. the \emph{lower natural mean $\FF$-entropy}, of a subset $X\subset A^G$ by 
     \[
     \overline{\ent}_\FF (X) = \limsup_{i \in I} \frac{|X_{F_i}|}{|F_i|}, \quad \quad \underline{\ent}_\FF (X) = \liminf_{i \in I} \frac{|X_{F_i}|}{|F_i|}. 
     \]
     \par 
     Similarly, we define respectively the \emph{upper Banach mean entropy} and the \emph{lower Banach mean entropy} of $X$ by 
     \[
     \overline{\Ent} (X) = \lim_{i \in I} \sup_{g \in G} \frac{|(gX)_{F_i}|}{|F_i|}, \quad \quad \underline{\Ent} (X) = \lim_{i \in I} \inf_{g \in G} \frac{|(gX)_{F_i}|}{|F_i|}. 
     \]
\end{definition}
\par 
Again, the Banach mean entropies do not depend on the F\o lner net. When $\overline{\ent}_\FF (X)=\underline{\ent}_\FF (X)$, resp. $\overline{\Ent} (X)=\underline{\Ent} (X)$, which holds for example when $X$ is a subshift, we simply denote ${\ent}_\FF (X)$, resp. ${\Ent} (X)$. 
\par 
We also have the following chain of inequalities as in Lemma~\ref{l:chain-densities}: 

\begin{lemma}
 Let $S$ be a subset of an amenable group $G$ with a F\o lner net $\FF$. Let $A \neq \varnothing$ be a  finite alphabet and let $X \subset A^G$. Then 
\[ 
  0 \leq \underline{\Ent} (X)\leq   \underline{\ent}_\FF (X)\leq   \overline{\ent}_\FF (X) \leq   \overline{\Ent} (X) \leq \log |A|. 
 \]
 Moreover, if $X\subset A^G$ is a subshift then $\underline{\Ent} (X)=\overline{\Ent} (X)$. \qed 
\end{lemma}

\subsection{Natural and Banach mean dimensions} 
\label{s:banach-mean-dim}

\begin{definition}
     Let $G$ be an amenable group with a F\o lner net $\FF=(F_i)_{i \in I}$. Let $A$ be a nontrivial finite dimensional vector space alphabet. We define respectively the \emph{upper natural mean $\FF$-dimension} and the \emph{lower  natural mean $\FF$-dimension} of a subset $X\subset A^G$ by 
     \[
     \overline{\mdim}_\FF (X) = \limsup_{i \in I} \frac{\dim X_{F_i}}{|F_i|\dim A}, \quad \quad \underline{\mdim}_\FF (X) = \liminf_{i \in I} \frac{\dim X_{F_i}}{|F_i|\dim A}. 
     \]
     \par 
     Similarly, we define respectively the \emph{upper Banach mean dimension} and the \emph{lower Banach mean dimension} of $X$ by 
     \[
     \overline{\mDim} (X) = \lim_{i \in I} \sup_{g \in G} \frac{\dim (gX)_{F_i}}{|F_i|\dim A}, \quad \quad \underline{\mDim} (X) = \lim_{i \in I} \inf_{g \in G} \frac{\dim (gX)_{F_i}}{|F_i|\dim A}. 
     \]
\end{definition}
\par 
Similarly to mean entropies, we have: 
\begin{lemma}
 Let $S$ be a subset of a countable amenable group $G$ with a F\o lner sequence $\FF$. Let $A$ be a nontrivial  finite-dimensional vector space. Then  for every  vector subspace $X \subset A^G$, we have 
 \[
  0 \leq \underline{\mDim} (X)\leq   \underline{\mdim}_\FF (X)\leq   \overline{\mdim}_\FF (X) \leq   \overline{\mDim} (X) \leq \dim  A. 
 \]
 Moreover, if $X\subset A^G$ is a linear subshift then $\underline{\mDim} (X)=\overline{\mDim} (X)$. \qed 
\end{lemma}
\par 
 Whenever $\overline{\mdim}_\FF (X)=\underline{\mdim}_\FF (X)$, resp. $\overline{\mDim} (X)=\underline{\mDim} (X)$,  we shall simply denote ${\mdim}_\FF (X)$, resp. ${\mDim} (X)$.

\subsection{Quasi-tilings of amenable groups} 
We say that $E\subset F$ is a \emph{$(1-\varepsilon)$-subset} of a finite set $F$ if $|E|>(1-\varepsilon) |F|$. 
Let $G$ be a countable amenable  group and let $\alpha, \beta \in [0,1]$. Then a collection $T=\{(T_i, T_i^o\}_{i \in I}$ is called an \emph{$\alpha$-disjoint} and \emph{$\beta$-covering quasi-tiling} of $G$ if: 
\begin{enumerate} [\rm (i)] 
    \item 
    $T_i^o$ is a $(1-\alpha)$-subset of $T_i$ for all $i \in I$;
    \item 
    $T_i^o \cap T_j^o = \varnothing$ for all $i \neq j$ in $I$;
    \item 
    $\underline{D}_G\left( \bigcup_{i \in I} T_i \right)\geq \beta$. 
\end{enumerate}

We shall need the following slightly more refined version of  the well-known  technical lemma \cite[Lemma~4.1]{tiling-amenable} (see also  \cite[I.\S2.  Theorem~6]{ornstein-weiss}). 
\par 
\begin{lemma}
\label{l:ornstein-weiss}
    Let $G$ be a countable amenable group with a F\o lner sequence $\FF=(F_n)_{n \geq 0}$ of symmetric sets containing the unity. Fix a finite subset  $M \subset G$  and $\varepsilon>0$. Then there exists an integer $r>0$ such that for every $n_0>0$, we can find an  $\varepsilon$-disjoint and $(1-\varepsilon)$-covering quasi-tiling $T=\{(T_i, T_i^o)\}_{i \in I}$ of $G$ with $r$ shapes $F_{n_1}, \dots, F_{n_r}$ where $n_0 <n_1 <\dots<n_r$ and such that $T_i^o= T_i^{-M}$ for every $i \in I$. 
\end{lemma}

\begin{proof}
    The proof is similar, \emph{mutatis mutandis}, to the proof of \cite[Lemma~4.1]{tiling-amenable} where we only need to require that $T_i^o= T_i^{-M}$ for all $i \in I$ at every step of the induction. Note that by properties of F\o lner sequences and as $M$ is a fixed finite subset of $G$, a similar construction of such quasi-tilings  as in the proof of \cite[Lemma~4.1]{tiling-amenable}  can be applied. 
\end{proof}

\begin{lemma}
    \label{l:technical-lem-density}
    Let $\alpha, \beta\in [0,1]$. Let $T=\{(T_i, T_i^o)\}_{i \in I}$ be a $(1-\alpha)$-disjoint and $\beta$-covering quasi-tiling  of  a countable amenable group $G$. Then the disjoint quasi-tiling $\{T_i^o\}_{i \in I}$ is an $\alpha\beta$-covering of $G$. 
\end{lemma}

\begin{proof}
    See the proof of \cite[Lemma~3.4]{tiling-amenable}
\end{proof}

\section{Almost injective cellular automata}

\label{s:almost-injective}

Following \cite{phung-weakly}, we say that a group $G$ is \emph{weakly surjunctive} if for every finite group alphabet $A$, any injective CA  $\tau \colon A^G \to A^G$ which is also a group homomorphism (also called a \emph{group CA}) must  surjective. A group $G$ is \emph{surjunctive} if for every finite  alphabet $A$, all injective CA $\tau \colon A^G \to A^G$ are surjective. Hence, sofic groups are surjunctive and the Gotschalk surjunctivity  conjecture states that every group is surjunctive. 
\par 
We begin with the following observation which shows the injectivity of a restriction of a group CA over an infinite group universe guarantees that the group CA is injective. 
\par 
\begin{theorem}
\label{t:almost-injective-group-ca}
Let $G$ be an infinite group and let $A$ be a group alphabet.  Suppose that $\tau \colon A^G \to A^G$ is a group CA such that $\tau\vert_{U}$ is injective where $U\subset A^G$ is a nonempty open subshift. Then $\tau$ is injective. If in addition $G$ is a  weakly surjunctive group and  $A$ is a finite group then $\tau$ is surjective. 
\end{theorem}

\begin{proof}
    Let $\Gamma = \Ker \tau \subset A^G$ then $\Gamma$ is a group subshift of finite type and thus closed. Suppose on the contrary that $\tau$ is not injective. Then  $\Gamma^*=\Gamma \setminus \{1_A^G\}$ is nonempty. Let us denote  
    $U \Gamma^* = \{uv \colon u \in U, \, v \in \Gamma^*\}$.  Since $\tau\vert_U$ is injective, we deduce that 
    \begin{equation}
    (U \Gamma^*) \cap U = \varnothing.  
    \end{equation}
    \par 
    Consequently,  $U  \Gamma^* \subset \Sigma=A^G \setminus U$. In particular, for an arbitrary $x \in \Gamma^*$, we deduce that $\Sigma$ contains the nonempty  open set $Ux$ of $A^G$. Since $\Sigma$ is a closed subshift and $G$ is infinite, it follows that $\Sigma$ is a dense closed subset of $A^G$ and we can thus conclude that $\Sigma= A^G$, which is a contradiction. 
Therefore, we must have $\Gamma = \{1_A^G\}$, i.e., $\tau$ is injective and the proof is thus complete. 
\end{proof}
\par 
Theorem~\ref{t:almost-injective-group-ca} clearly fails for CA over finite universes.  
In fact, the following example shows that  the implication $\tau\vert_U$ injective$\implies$$\tau$ injective fails with $U$ being only an open  subset (and not an open subshift). 
\begin{example}
    Let $G= \Z$ and let $A=\{0,1\}$. Let $\tau \colon A^G \to A^G$ be a CA with the  local rule at each $i \in \Z$ given by $f(x_i, x_{i+1})=x_i +  x_{i+1}$ (mod 2). Let $\Sigma = \{ x \in A^G \colon x_0 = 1\} $. 
    The complement of $\Sigma$ is $U= \{ x \in A^G \colon x_0 = 0\}$ and the linear $\tau$ is clearly injective when restricted to $U$. However, $\tau$ is not injective since its kernel contains the two  constant configurations. Nevertheless, Theorem~\ref{t:finite-z-d-almost-general-closed-subset-dense-periodic} or a direct computation implies that $\tau$ must be surjective. The same holds if we exchange the role of $\Sigma$ and $U$. 
\end{example}

We obtain the following  result which allows us to bound the period of periodic configurations in the approximation of a strongly irreducible subshift of finite type.

\begin{theorem}
\label{t:density-periodic-general}
    Let $G= \Z^d$ and let $A$ be a set. Let $X\subset A^G$ be a strongly $\Delta$-irreducible subshift of finite type with window $\Delta = [-r,r]^d\cap \Z^d$. Suppose that $X$ contains a periodic configuration $x_0$ of total period $2n_0>2r$, i.e., $2n_0\Z^d x_0 = x_0$. Then for $\FF=(F_n)_{n \geq 1}$ where $F_n= [-k_n, k_n]^d \cap \Z^d$ with $k_n=(nn_0-2r)(n_0+1)$, we have:
    \[
    (\Fix((2k_n+4r)G) \cap X )\vert_{F_n} = X_{F_n}. 
    \]
\end{theorem}

\begin{proof}
    The result follows from an immediate modification of the proof of \cite[theorem 1.1]{csc-nonlinearity}. With the same notations as in the proof of \cite[theorem 1.1]{csc-nonlinearity} and $G=\Z^d$, we have $L\subset K$ so that $H=K \cap L= L$ where $L=(2k_n+4r)G \subset K= 2n_0G$, $\Omega=F_n$, and  $\Omega_i= [-k_n-ir, k_n+ir]^d \cap \Z^d$.  
\end{proof}

Hence, we obtain the following strengthened form of surjunctivity for CA over strongly irreducible closed subshifts of finite type. 

\begin{theorem}
\label{t:finite-z-d-almost-general-closed-subset-dense-periodic}
 Let $G=\Z^d$, $d \geq 1$, and let $A$ be a finite alphabet. Let $X \subset A^G$ be 
    {strongly irreducible closed subshift of finite type} which contains a periodic configuration.  
    Suppose that $\tau \colon X \to X$ is a  CA such that $\tau\vert_{U}$ is injective where $U \subset X$ is a {nonempty open subset}. Then $\tau$ is surjective and pre-injective. 
\end{theorem}

\begin{proof}
\par 
Let us denote $\Gamma = \tau(X) \subset X$. 
Since $U\subset X$ is a nonempty open subset,  $U$ contains a cylinder $C(p)=\{ x \in X  \colon x \vert_S=p\}$ for some pattern $p \in X_S$ where $S\subset G$ is a finite subset.
\par 
Choose a sufficiently large window $\Delta = [-r,r]^d\cap \Z^d$, where $r \geq 2$ is an integer, such that $\Delta \supset S$ is a memory set of $\tau$ and $X= \Sigma(A^G; X_\Delta, \Delta)$  is a strongly $\Delta$-irreducible subshift of finite type. 
In particular, we deduce from the definition of $\Delta$-irreducibility that for every finite subset $E \subset G$, we have: 
\begin{equation}
    \label{e:finite-z-d-almost-general-closed-subset-dense-periodic-1}
    C(p)_{E \setminus S\Delta} = X_{E \setminus S\Delta}.
\end{equation}
\par
Indeed, let $\tilde{p}\in X$ be any configuration extending $p$. Let $x \in X$. Since $S\Delta \cap ( E \setminus S\Delta ) = \varnothing$ and $X$ is $\Delta$-irreducible, there exists $y \in X$ such that $y\vert_{S} = \tilde{p}\vert_S$ and $y \vert_{E \setminus S\Delta} = x\vert_{E \setminus S\Delta}$. Hence, $X_{E \setminus S\Delta}\subset C(p)_{E \setminus S\Delta}$. As clearly $C(p)_{E \setminus S\Delta} \subset X_{E \setminus S\Delta}$, the relation \eqref{e:finite-z-d-almost-general-closed-subset-dense-periodic-1} follows. 
\par 
Fix a periodic configuration  in $X$ of total period $2n_0>2r$. 
Let us denote  $\FF=(F_n)_{n \geq 3}$ where $F_n= [-k_n, k_n]^d \cap \Z^d$ and with  $k_n=(nn_0-2r)(n_0+1)$. Note that $k_n\geq r(n_0+1)\geq 4r$ for all $n \geq 3$ and $\FF$ is a F\o lner sequence of $\Z^d$. 
\par 
Let $H_n =  (2k_n+4r)\Z^d$ then by Theorem~\ref{t:density-periodic-general}, we have 
\begin{equation}
\label{e:finite-z-d-almost-general-closed-subset-dense-periodic-2}
    (\Fix(H_n) \cap X)_{F_n} = X_{F_n} 
\end{equation}

where $P_n = \Fix H_n = \{x \in A^G \colon hx=x,\, \forall h \in H_n\}\subset A^G$ is the space of $H_n$-periodic configurations. Note that we have a canonical bijection $P_n \simeq A^{F_n}$ given by the restriction $x \mapsto x\vert_{F_n}$.  
\par 
Let us define $Q_n = P_n \cap C(p) =  \{x \in P_n\cap X \colon x\vert_{S} =p \}$ for every $n \geq 3$. Then $Q_n \subset U$ and we thus obtain induced injective maps 
$\tau \vert_{Q_n} \colon Q_n \to P_n$  for all $n \geq 3$. 
By \eqref{e:finite-z-d-almost-general-closed-subset-dense-periodic-2}, we have $(P_n \cap X)_{F_n}= X_{F_n}$. Since $F_n \setminus S\Delta \subset F_n$ and $F_n$ is finite, we deduce from  $(P_n \cap X)_{F_n}= X_{F_n}$ and the relation \eqref{e:finite-z-d-almost-general-closed-subset-dense-periodic-1} that: 
\begin{align}
\label{e:finite-z-d-almost-general-closed-subset-dense-periodic-3} 
(P_n \cap X)_{F_n \setminus S\Delta} = X_{F_n \setminus S\Delta}=C(p)_{F_n \setminus S\Delta}. 
\end{align} 
\par 
We claim that 
\begin{align}
\label{e:finite-z-d-almost-general-closed-subset-dense-periodic-4}
(Q_n)_{F_n \setminus S\Delta} = (P_n \cap C(p))_{F_n \setminus S\Delta} = X_{F_n \setminus S\Delta}=C(p)_{F_n \setminus S\Delta}. 
\end{align}
\par 
Indeed, we clearly have $(P_n \cap C(p))_{F_n \setminus S\Delta} \subset (P_n \cap X)_{F_n \setminus S\Delta} = X_{F_n \setminus S\Delta}$ by \eqref{e:finite-z-d-almost-general-closed-subset-dense-periodic-3} as $C(p) \subset X$.  On the other hand, let us fix a $H_n$-invariant configuration $x \in P_n \cap X$ and let $q = x_{F_n \Delta^2\setminus S\Delta} \in  X_{F_n \Delta^2 \setminus S\Delta}=  C(p)_{F_n \Delta^2 \setminus S\Delta}$ (cf.~\eqref{e:finite-z-d-almost-general-closed-subset-dense-periodic-1}). Thus, we can find a configuration $y \in C(p)$ which extends $q$, that is, 
\[
y \in X, \quad  y\vert_S = p \quad \text{and}\quad y \vert_{F_n \Delta^2 \setminus S\Delta}=q= x\vert_{F_n \Delta^2 \setminus S\Delta}. 
\] 
\par
Consider the patched configuration $z \in A^G$ defined by $z\vert_{G\setminus (H_n S\Delta)} = x\vert_{G\setminus (H_n S\Delta)}$ and $z(h s \delta)= y(s\delta) $ for all $h \in H_n$, $s \in S$ and $\delta \in \Delta$. Since $x$ is $H_n$-invariant, we deduce from the construction that so is $z$. In other words, $z \in P_n$. Moreover, note that 
$S\Delta^2 \subset F_n \Delta^2$ and $F_n\Delta^2$ contains a  fundamental domain of $H_n$.  Consequently, 
as $x \in X$ is $H_n$-invariant and $y\vert_{F_n \Delta^2} \in X_{\vert_{F_n \Delta^2}}$ and $ y \vert_{F_n \Delta^2 \setminus S\Delta}= x\vert_{F_n \Delta^2 \setminus S\Delta}$, we deduce that $z\vert_{g\Delta} \in X_{g \Delta}$ for all $g \in G$. Hence, $z \in X$ since  $\Delta$ is a defining window of the subshift of finite type $X$. As $z_S=y_S=p$, it follows that $z \in C(p)$. To summarize, we have shown that  
$z \in P_n\cap C(p)$ and $z\vert_{F\setminus S\Delta} =x\vert_{F\setminus S\Delta}$. The claim \eqref{e:finite-z-d-almost-general-closed-subset-dense-periodic-4} is thus proved.   
\par 
We infer from the injectivity of $\tau\vert_{Q_n}$, the relation  \eqref{e:finite-z-d-almost-general-closed-subset-dense-periodic-4}, and the inclusion  $X_{F_n} \subset A^{S\Delta} \times X_{F_n\setminus S \Delta}$  the following estimation:  
\begin{align*}
\vert \Gamma_{F_n\Delta^2}\vert \geq \vert \tau(Q_n) \vert = |Q_n| \geq |(Q_n)_{F_n \setminus S\Delta}| = |X_{{F_n \setminus S\Delta}}| \geq \frac{|X_{F_n}|}{|A|^{|S\Delta|}}. 
\end{align*} 
\par 
Therefore, we obtain
\begin{align*}
\overline{\ent}_\FF (\Gamma) & = \limsup_n \frac{\log \vert \Gamma_{F_n\Delta^2}\vert}{\vert F_n \Delta^2 \vert}  \\
& \geq   \limsup_n \frac{\log |X_{{F_n}}|}{\vert F_n \Delta^2\vert}  - \limsup_n \frac{|A|^{|S\Delta|}}{|F_n \Delta^2|}\\
& =   \limsup_n \frac{\log |X_{{F_n}}|}{\vert F_n\vert} \limsup_n \frac{|F_n|}{|F_n \Delta^2|} - 0\\ 
& =  \limsup_n \frac{\log |X_{{F_n}}|}{\vert F_n\vert} \\ 
&= \overline{\ent}_\FF (X). 
\end{align*}
\par 
Hence, $\Gamma$ is a closed subshift of  maximal mean entropy of a strongly irreducible subshift of finite type $X$. It follows that $\Gamma = X$ by \cite[Proposition~4.2]{csc-myhill-strongly-irreducible}. In other words, $\tau$ is surjective. By the Garden of Eden theorem for strongly irreducible subshifts of finite type \cite{fiorenzi}, we deduce that $\tau$ is pre-injective.  
\end{proof}

 \section{Almost pre-injective non-uniform cellular automata}
\label{s:myhill}

 \subsection{Generalized Myhill property}

We begin with the following lemma. 

\begin{lemma}
\label{l:entropy-open-max}
Let $G$ be an infinite amenable group with a given F\o lner net $\FF=(F_i)_{i \in I}$ and let $A$ be a finite alphabet. Then for all subsets $X, Y \subset A^G$ and every finite subset $S \subset G$ such that $X\vert_{G \setminus S} = Y \vert_{G \setminus S}$, we have
\begin{equation*}
    \overline{\ent}_\FF X = \overline{\ent}_\FF Y. 
\end{equation*} 
\par 
In particular, for every nonempty open subset $U \subset A^G$, we have 
\[
\ent_\FF U = {\ent}_\FF A^G = \log |A|.
\] 
\end{lemma}

\begin{proof}
   We can suppose without loss of generality that $S \subset F_i$ for all $i \in I$. Since $G$ is infinite, we have 
    $\lim_{I} |F_i| = \infty$ and thus $\limsup_I \frac{|S|}{|F_i|}=0$. 
Let $Z=X\vert_{G \setminus S} = Y\vert_{G \setminus S} \subset A^{G \setminus S}$. Fix a pattern $p \in A^S$ and let us consider $\Gamma = \{p\} \times Z$ and $\Lambda = A^S \times Z$.
\par 
Then for every $i \in I$, we have $\Gamma_{F_i} =\{p\} \times Z_{F_i \setminus S}$ and 
$\Lambda_{F_i} = A^S \times Z_{F_i \setminus S}$.  
On the one hand, we have $X_{F_i}, Y_{F_i} \subset \Lambda_{F_i}$ for every $i \in I$ and thus  by the definition of the mean entropy,  
\[
\overline{\ent}_\FF X = \limsup_I \frac{\log |X_{F_i}|}{|F_i|}\leq  \limsup_I \frac{\log |\Lambda_{F_i}|}{|F_i|} = \overline{\ent}_\FF \Lambda
\]
and similarly $\overline{\ent}_\FF Y \leq \overline{\ent}_\FF \Lambda$. On the other hand, we find that: 
 \begin{align*}
\overline{\ent}_\FF X & = \limsup_I \frac{\log |X_{F_i}|}{|F_i|}   \geq \limsup_I \frac{\log |X_{F_i\setminus S}|}{|F_i|}\\
& = \limsup_I \frac{\log |Z_{F_i \setminus S}|}{|F_i| }   =  \limsup_I \frac{\log |\Gamma_{F_i}|}{|F_i|}  = \overline{\ent}_\FF \Gamma.  
    \end{align*}
Similarly, $\overline{\ent}_\FF Y \geq \overline{\ent}_\FF \Gamma$. However, we have: 
\begin{align*}
\overline{\ent}_\FF \Lambda & = \limsup_I \frac{\log |\Lambda_{F_i}|}{|F_i|}   = \limsup_I \frac{\log |A^S||Z_{F_i}|}{|F_i|} \\
 & =  \limsup_I \frac{\log |Z_{F_i}|}{|F_i|} +  \limsup_I \frac{|S|}{|F_i|} \log |A|  \\
 &=  \limsup_I \frac{\log |\Gamma_{F_i}|}{|F_i|} + 0
 = \overline{\ent}_\FF \Gamma. 
\end{align*} 
    \par 
    Consequently, we have 
    \[
    \overline{\ent}_\FF \Gamma \leq \overline{\ent}_\FF X,\,  \overline{\ent}_\FF Y \leq \overline{\ent}_\FF \Lambda = \overline{\ent}_\FF \Gamma. 
    \]
    We can thus conclude that $ 
   \overline{\ent}_\FF X= \overline{\ent}_\FF Y$. 
For the last assertion, note that every nonempty open subset $U\subset A^G$ contains a cylinder $C(p)= \{x \in A^G \colon x\vert_S= p\}$ where $p \in A^S$ is a finite pattern, i.e., $S\subset G$ is a finite subset. Then  $C(p) \subset U \subset A^G$ thus $\ent C(p) \leq \ent U \leq \ent A^G$. The first assertion applied for $X= C(p)$, $Y=A^G$ shows that 
$\ent C(p)= \ent A^G= \log |A|$. Hence, we deduce that $\ent U= \ent C(P)= \log |A|$ which ends the proof. 
\end{proof}
   \par 
We establish the following general result which generalizes the Myhill property.  

\begin{theorem}
\label{t:GOE-nuca-c-1}
    Let $S \subset G$ be a subset of  an amenable group $G$ and let $A$ be a finite alphabet. 
    Let $p \in A^S$ and $U=\{p\} \times A^{G \setminus S}$. 
    Suppose that $\tau \colon A^G \to A^G$ is an NUCA with finite memory such that $\tau\vert_{U}$ is pre-injective. Then $\overline{\ent}_\FF (\im \tau\vert_{U} ) \geq (1 - \overline{d}_\FF(S)) \log |A|$ for every F\o lner net $\FF$ of $G$.  
\end{theorem}

\begin{proof} 
Our proof generalizes the proof in the case of pre-injective CA over the full shift. 
 Let us choose  a memory set $M\subset G$ of $\tau$ such that $M=M^{-1}$ and $1_G \in M$.  Let $\FF=(F_i)_{i \in I}$ be a F\o lner net of $G$ and we can suppose without loss of generality that  $SM \subset F_i$ for all $i \in I$. \par 
 Since $\overline{\ent}_\FF(\im \tau\vert_{U} ) \geq 0$ and $\overline{d}_\FF(S) \in [0,1]$, we can suppose without loss of generality that $\overline{d}_\FF(S)<1$. 
\par 
Let us denote $\Gamma  =  \im \tau\vert_U $ as the image of $U$ under $\tau$. 
Suppose on the contrary that $\overline{\ent}_\FF \Gamma < (1-\overline{d}_\FF(S)) \log |A|$.
Note that for all $i\in I$ we have: 
\[ 
 \Gamma_{F_iM} \subset \Gamma_{F_i} \times A^{\partial_M F_i}. 
\] 
Therefore, $| \Gamma_{F_iM} |\leq |\Gamma_{F_i}| |A^{\partial_M F_i}|$ and we find that: 
\begin{align} 
\label{e:t:GOE-nuca-c-1-a}
\frac{\log |\Gamma_{F_iM}|} {|F_i\setminus S|} & \leq \frac{\log |\Gamma_{F_i}|}{|F_i\setminus S|}+\frac{|{\partial_M F_i}|\log |A|}{|F_i\setminus S|}.
\\
& \leq \frac{|F_i|}{|F_i\setminus S|}\frac{\log |\Gamma_{F_i }|}{|F_i|}+\frac{|{\partial_M F_i}|}{|F_i|} \frac{|F_i|}{|F_i\setminus S|} \log |A|.\nonumber
\end{align} 
 \par 
Note that since $G$ is an  amenable group,
 \[
 \lim_{i \in I } \frac{|{\partial_M F_i}|}{|F_i|}  =0.
 \]
 \par
 Moreover, by the definition of $\overline{d}_\FF(S)$,  we have: 
 \[
  \limsup_{i \in I}  \frac{|F_i|}{|F_i\setminus S|} =   \limsup_{i \in I}  \frac{1}{\frac{|F_i| - |F_i\cap  S|}{|F_n|}}= \frac{1}{1- \limsup_{i \in I}  \frac{|S\cap F_i|}{|F_i|}} =\frac{1}{1-\overline{d}_\FF(S)}.  
 \]
 \par 
 Hence, by passing to the limit the inequality \eqref{e:t:GOE-nuca-c-1-a} for $i \in I$ in the index set $I$ and by applying the hypothesis $\overline{\ent}_\FF (\im \tau\vert_{U} ) < (1 - \overline{d}_\FF(S)) \log |A|$, we obtain:   
 \begin{align*}
     \limsup_{i\in I} \frac{\log |\Gamma_{F_iM}|}{|F_i \setminus S|}\leq \frac{1}{1-\overline{d}_\FF(S)} \limsup_{i \in I} \frac{\log |\Gamma_{F_i }|}{|F_i|} = \frac{1}{1-\overline{d}_\FF(S)} \overline{\ent}_\FF \Gamma < \log |A|. 
 \end{align*}
\par 
Consequently, there exists  $k \in I$ such that 
\[
\frac{\log |\Gamma_{F_k M}|}{|F_k\setminus S|} < \log |A|
\]
 thus 
\[
|\Gamma_{F_kM }| < |A|^{|F_k\setminus S|}. 
\]
\par 
Let $a \in A$ an arbitrary element and let $Z= \{p\} \times A^{F_k\setminus S} \times \{a\}^{G \setminus (S\cup F_k)}$.  Then $|Z| = | A^{F_k\setminus S}|=|A|^{ {|F_k\setminus S|}}$ and we have: 
\begin{align*}
    |\tau(Z)| = |\tau(Z)_{F_kM}| \leq |\Gamma_{F_kM}| < |A|^{|F_k\setminus S|} = |Z|. 
\end{align*}
\par 
We can thus find two distinct configurations $z_1$ and $z_2$ in $Z$ such that $\tau(z_1)=\tau(z_2)$. In particular, $\tau$ is not pre-injective even when restricted to $U$, which is a contradiction to the hypothesis of the theorem that $\tau\vert_U$ is pre-injective. Hence, we must have $\overline{\ent}_\FF \Gamma \geq (1-\overline{d}_\FF(S)) \log|A|$ and the proof is complete.  
\end{proof}

\par 
As a consequence, we obtain the following surjunctivity-like result which strengthens the Myhill property of CA over   amenable group universes. 
\par 
\begin{corollary}
Let $G$ be an amenable group and let $A$ be a finite alphabet. Let $S\subset G$ with $d_G(S)=0$. 
Let $p \in A^S$ and $U=\{p\} \times A^{G \setminus S}$. Suppose that $\tau \colon A^G \to A^G$ is a CA such that $\tau\vert_{U}$ is pre-injective. Then $\tau$ is surjective and pre-injective. 
\end{corollary}

\begin{proof}
 Since $d_G(S)=0$ and $\tau\vert_U$ is pre-injective by hypotheses,  we infer from Theorem~\ref{t:GOE-nuca-c-1} that for every F\o lner net $\FF$ of $G$: 
    \[
    \log |A| = \overline{\ent}_\FF (A^G) \geq \overline{\ent}_\FF (\im \tau ) \geq \overline{\ent}_\FF (\im \tau\vert_{U}  ) \geq \log |A|.
    \]
    \par 
    Consequently, $\overline{\ent}_\FF (\im \tau  ) = \log |A|$.   Therefore, the classical Garden of Eden theorem for CA implies that $\tau$ is surjective and pre-injective. The proof is thus complete. 
 \end{proof}

\par 
The next corollary implies  that the image of every sufficiently pre-injective NUCA over a nonempty open subset must have maximal mean entropy.  

\begin{corollary}
\label{l:GOE-nuca-c-1}
    Let $G$ be an infinite amenable group and let $A$ be a finite alphabet. Suppose that $\tau \colon A^G \to A^G$ is an NUCA with finite memory such that $\tau\vert_{U}$ is pre-injective where $U \subset A^G$ is a nonempty open subset. Then $\overline{\ent}_\FF (\im \tau\vert_{U} ) = \log |A|$ for every F\o lner net $\FF$ of $G$. 
\end{corollary}

\begin{proof}
Since $U$ is an open of $A^G$ in the prodiscrete topology, it contains a cylinder $C(p)=\{x \in A^G\colon x_S=p\}$ for some pattern $p \in A^S$ where $S\subset G$ is a finite subset. It follows that $\tau\vert_{C(p)}$ is pre-injective. On the other hand, $C(p)=\{p\} \times A^{G \setminus S}$ and we clearly have $\overline{d}_\FF(S)=0$ since $S$ is finite and $G$ is infinite so that $\lim_{i \in I} |F_i|=\infty$ for any fixed F\o lner net $\FF=(F_i)_{i \in I}$ of $G$. We can thus infer from Theorem~\ref{t:GOE-nuca-c-1} that 
    \[
    \log |A| = {\ent} (A^G) \geq \overline{\ent}_\FF (\im \tau\vert_{U} ) \geq \overline{\ent}_\FF (\im \tau\vert_{C(p)} ) \geq \log |A|.
    \]
    \par 
    Therefore, $\overline{\ent}_\FF (\im \tau\vert_{U} ) = \log |A|$ and the proof is complete. 
 \end{proof}
\par 
 
By specializing to the class of CA, we deduce the following result. 

\begin{corollary} 
Let $G$ be an infinite amenable group and let $A$ be a finite alphabet. Suppose that $\tau \colon A^G \to A^G$ is a CA such that $\tau\vert_{U}$ is pre-injective where $U\subset A^G$ is an open subset. Then $\tau$ is surjective. 
\end{corollary} 

\begin{proof}
    By Lemma~\ref{l:GOE-nuca-c-1}, we have $\overline{\ent}_\FF(\im \tau\vert_U)=\log|A|$  for any F\o lner net $\FF$ of $G$. Since $\im \tau\vert_U \subset \im \tau \subset A^G$ and $\ent A^G= \log |A|$, we deduce from the monotonicity of the mean entropy that 
    $\overline{\ent}_\FF \im \tau = \log|A|$. Hence, the classical Garden of Eden theorem for CA implies that $\tau$ is surjective. 
\end{proof}

\subsection{Generalized Moore property for NUCA}

Conversely, we have: 

\begin{lemma}
\label{l:GOE-nuca-c-2}
    Let $G$ be an infinite amenable group and let $A$ be a finite alphabet.  Let $\tau \colon A^G \to A^G$ be an NUCA with finite memory whose configuration of the local transition maps is asymptotically constant. Suppose that $\overline{\ent}_\FF (\im \tau ) = \log |A|$ for some F\o lner net $\FF$ of $G$. Then there exists an open subset $U\subset A^G$ such that $\tau\vert_U$ is pre-injective. Moreover,  $\im \tau$ contains an open subset of $A^G$. 
\end{lemma}

\begin{proof}
    Let  $\sigma \colon A^G \to A^G$ be a CA whose local transition map coincides with the local transition maps of $\tau$ outside of some finite subset $E \subset G$. Consider the cylinder $U= \{p\} \times A^{G \setminus EM} $ where $M$ is a memory set of $\tau$ such that $M=M^{-1}$,  $1_G \in M$, and $p \in A^{EM}$ is an arbitrary pattern. 
    \par 
 Then by Lemma~\ref{l:entropy-open-max}, we have  $\overline{\ent}_\FF \im \sigma = \overline{\ent}_\FF  \im \tau$ since 
$(\im \sigma )_{G \setminus EM} = (\im \tau )_{G \setminus EM}$. Therefore, $\overline{\ent}_\FF \im \sigma = \log |A|$ and thus by the Garden of Eden theorem, we deduce that  $\sigma$ is surjective and pre-injective. Consequently, $\tau$ is also pre-injective when restricted to the open subset $U \subset A^G$. This proves the first assertion. 
 \par
 For the second assertion, note that $\im (\tau)_{G \setminus EM^2} = \im (\sigma)_{G \setminus EM^2} = A^{EM^2}$ since $\sigma$ is surjective. For every $q \in A^{EM^2}$, we define  
 \[
 Z_q = \im (\tau)  \cap (\{q\} \times A^{G \setminus EM^2}).
 \]
 \par 
 In other words, $Z_q$ is the set of configurations in the image of $\tau$ whose restriction to $EM^2$ is exactly the pattern $q$. Since $\im (\tau)$ is a closed subset of $A^G$ by the closed image property, we deduce that $Z_q$ and $(Z_q)_{G\setminus EM^2}$ are respectively closed (in fact compact) subsets  of $A^G$ and $A^{G \setminus EM^2}$. 
 \par 
Clearly we have the relation $\im(\tau)=\bigcup_{q \in A^{EM^2}} Z_q$ and it follows that 
 \[
 \bigcup_{q \in A^{EM^2}} (Z_q)_{G\setminus EM^2} = \im (\tau)_{G \setminus EM^2}  = A^{G \setminus EM^2}. 
 \]
 \par 
 By the Baire category theorem applied to the non-empty complete metric space $A^{G\setminus EM^2}$, there must exist $q \in A^{EM^2}$ such that $(Z_q)_{G\setminus EM^2}$ contains a cylinder $V$ of  $A^{G\setminus EM^2}$ in the prodiscrete topology. We conclude  by our definition of the set $Z_q$ that $\im(\tau)$ contains the open cylinder $V \times \{q\} \subset A^G$. The proof is thus complete. 
\end{proof}

\par 
The following example shows that the above Lemma~\ref{l:GOE-nuca-c-2} is optimal in the sense that it fails if we remove the hypothesis  that $\tau$ is a local perturbation of a CA.  

\begin{example}
    Let $G= \Z$ and let    
    $\tau \colon \{0,1\}^G \to \{0,1\}^G$ be an NUCA such that for all $x \in A^G$ and $n \in \Z$, we have $\tau(x)(n)= 0$ if $n$ is a perfect square and $\tau(x)(n)=x(n)$ otherwise. A direct computation shows that $\ent (\im \tau)= 1= \log |\{0,1\}|$. However, $\tau$ is not pre-injective when restricted to any cylinder subset of $\{0,1\}^G$. 
\end{example}

\subsection{Garden of Eden theorem  for the class $\mathrm{NUCA}_c$} 
\par 
By combining Lemma~\ref{l:GOE-nuca-c-1} 
and Lemma~\ref{l:GOE-nuca-c-2}, we can establish the following version of the Garden of Eden theorem for NUCA which are local perturbations of CA over infinite amenable group universes. 
\begin{theorem}
\label{t:GOE-nuca-c-v1} 
Let $G$ be an infinite amenable group and let $A$ be a finite alphabet. Let $\tau \colon A^G \to A^G$ be an NUCA with finite memory and whose configuration of the local transition maps is asymptotically constant. Then the following are equivalent: 
\begin{enumerate}[\rm (a)] 
\item  $\tau\vert_U$ is pre-injective for some nonempty open subset $U\subset A^G$. 
    \item $\overline{\ent}_\FF (\im \tau ) = \log |A|$ for some (or equivalently any)  F\o lner net $\FF$ of $G$. 
    \item  $\im \tau$ contains a nonempty open subset of $A^G$. 
\end{enumerate}
\end{theorem}

\begin{proof}
Suppose first that (a) holds for some open subset $U \subset A^G$. Then 
$\overline{\ent}_\FF (\im \tau\vert_U ) = \log |A|$ by Lemma~\ref{l:GOE-nuca-c-1}. Since 
$\im ( \tau\vert_U )\subset \im (\tau) \subset A^G$, the monotonicity of the mean entropy implies that: 
\[
\log |A| = \overline{\ent}_\FF \im ( \tau\vert_U )\leq \overline{\ent}_\FF \im (\tau) \leq \ent A^G = \log |A|. 
\]
\par 
Consequently, $\overline{\ent}_\FF \im (\tau)= \log|A|$ and we have proved that (a)$\implies$(b). The implications (b)$\implies$(a) and (b)$\implies$(c) result directly from Lemma~\ref{l:GOE-nuca-c-2}. Suppose that $\im (\tau)$ contains an open subset of $A^G$ then since $G$ is infinite, we have $\overline{\ent}_\FF \im(\tau) = \log |A|$ by Lemma~\ref{l:entropy-open-max}. Hence, (c)$\implies$(b) and the proof is complete. 
\end{proof}

\begin{proof}[Proof of Theorem~\ref{t:GOE-nuca-c-v1-intro}] 
The case when $G$ is infinite results from Theorem~\ref{t:GOE-nuca-c-v1}. When $G$ is finite, the properties (i) and (ii) in Theorem~\ref{t:GOE-nuca-c-v1-intro} hold true trivially. The proof of Theorem~\ref{t:GOE-nuca-c-v1-intro} is thus complete. 
\end{proof}

\section{Garden of Eden theorem for linear NUCA}
\label{s:GOE-linear-NUCA}
The Garden of Eden theorem for linear CA over the full shifts was first proved in  \cite{csc-goe-linear}. It was known to characterize amenable groups \cite{bartholdi-kielak} and the limit set of a CA \cite{cscp-jpaa}.  
\par 
For linear NUCA which are uniform outside a zero-density subset of cells, we obtain the following converse of Theorem~\ref{t:GOE-nuca-c-1}. 

\begin{theorem}
\label{t:linera-GOE-nuca-mdim-pre}
let $K$ be a field and let $M$ be a finite subset of an infinite countable amenable group $G$. Fix $s,c \in \LL(K^M, K)^G$ with $c$  constant and $s$ differs from $c$ only on a subset of $G$ of zero upper Banach  density.  Suppose that $\underline{\mDim} (\im \sigma_s ) > 1-d$ for $d\in ]0,1]$.   Then there exists $U=\{0\}^S \times K^{G \setminus S}$ where  $S\subset G$ with $\overline{D}_G(S)\leq d$ such that $\sigma_s\vert_U$ is pre-injective. 
\end{theorem}

\begin{proof}
For the countable amenable group $G$, let us fix an increasing F\o lner sequence $\FF=(F_n)_{n \geq 0}$ of symmetric sets containing the unity $1_G$. Let us fix also a symmetric memory set $M \subset G$ of $\tau$ such that $1_G \in M$. 
We denote $\Delta = \{g \in G \colon s(g) \neq c(g) \}\subset G$ then $\overline{D}_G(\Delta)=0$ by hypothesis.    
\par 
For the notations, let $\Gamma = \im \sigma_s \subset A^G$ and $\varepsilon = \mdim \Gamma - (1-d)$ then $\varepsilon >0$ by hypothesis. 
Since $G$ is an infinite amenable group, we have:
\[
\lim_{n \to \infty} \frac{|\partial_{M^2} F_n|}{|F_n|}=0 \quad \text{and} \quad \lim_{n \to \infty} |F_n| = \infty
\]
and $(F_n^{-M^2})_{n \geq 0}$ is also a F\o lner sequence of $G$. Moreover, 
\[
\varepsilon + (1-d) =\underline{\mdim} \Gamma = \liminf_{n \to \infty} \frac{\dim \Gamma_{F_n}}{|F_n|}
\]
and 
\[
\liminf_{n \to \infty} \sup_{g \in G}\frac{|\Delta M \cap gF_n|}{|F_n|} = \overline{D}_G(\Delta M)=0.
\] 
\par 
Therefore, there exists $n_0 \in \N$ such that for all $n \geq n_0$, we have 
\begin{align}
\label{e:linera-GOE-nuca-mdim-pre-1-b}
 \frac{|\partial_{M^2} F_n|}{|F_n|} <\frac{\varepsilon}{8}, 
\end{align} 

\begin{align}
\label{e:linera-GOE-nuca-mdim-pre-1-c}
\frac{\dim \Gamma_{F_n^{-M^2}}}{|F_n|} \geq \frac{\varepsilon}{2} + 1-d, \quad \text{and}  
\end{align} 

\begin{align}
    \label{e:linera-GOE-nuca-mdim-pre-1-d} 
  \frac{|\Delta  \cap gF_n|}{|F_n|} < \frac{\varepsilon}{16} \quad \text{for all }  g \in G. 
\end{align}
\par 
By  Lemma~\ref{l:ornstein-weiss},  there exist an integer $r>0$ and  an ${\varepsilon}$-disjoint and $\left(1-\frac{\varepsilon}{100}\right)$-covering quasi-tiling $T=\{(T_i, T_i^o)\}_{i \in I}$ of $G$ with $r$ shapes $F_{n_1}, \dots, F_{n_r}$ such that $n_0 <n_1 <\dots<n_r$ and $T_i^o=T_i^{-M}$ for every $i \in I$. 
In particular, $T_i^o \cap T_j^o= \varnothing$ for all $i \neq j$ in $I$. 
For every $i \in I$, let $F_{k_i}$ be the shape of $T_i$ where $k_i \in \{n_1, \dots n_r\}$ and let $g_i \in G$ be the center of $T_i$, i.e., 
\[
T_i=g_iF_{k_i}.
\] 
\par 
Recall that for every subset $F \subset G$, we have $\partial_{M^2} F= FM^2 \setminus F^{-M^2}$ where $F^{-M^2} = \{g \in F\colon gM^2 \subset F\}$ since $1_G\in M$ and $M=M^{-1}$. Hence, for every $n \geq n_0$, we infer from the inequality \eqref{e:linera-GOE-nuca-mdim-pre-1-b} that 
\begin{equation}
\label{e:linera-GOE-nuca-mdim-pre-1-a}
|F_n^{-M^2}| \geq |F_n| - |\partial_{M^2} F_n| \geq \left(1-\frac{\varepsilon}{8} \right) |F_n|. 
 \end{equation} 
\par 
Note that $M\subset M^2$ as $1_G \in M$. Hence, we have $F_{k_i}^{-M^2} \subset F_{k_i}^{-M}$ for every $i \in I$ by definition because  $gM^2 \subset F_{k_i}$ implies $gM \subset gM^2 \subset F_{k_i}$. Consequently, it holds for every $i \in I$ that:    
\begin{align}
    \label{e:linera-GOE-nuca-mdim-pre-1-f}
    E_i = T_i^o \cap g_i F_{k_i}^{-M^2} = g_iF_{k_i}^{-M} \cap g_i F_{k_i}^{-M^2} =  g_i F_{k_i}^{-M^2} = T_i^{-M^2}. 
\end{align}
\par 
It follows that $E_i M \subset g_iF_{k_i}^{-M} = T_i^o$ for all $i \in I$. Moreover, for all $i \neq j$ in $I$, we have $T_i^o \cap T_j^o= \varnothing$ by hypothesis and  thus $E_iM \cap E_jM=\varnothing$. 
\par 
On the other hand, we deduce from \eqref{e:linera-GOE-nuca-mdim-pre-1-d} and \eqref{e:linera-GOE-nuca-mdim-pre-1-a} that for all $i \in I$: 
\begin{align}
\label{e:linera-GOE-nuca-mdim-pre-1-g}
 |E_i \setminus (\Delta \cup g_i \Delta) | & \geq   |g_iF_{k_i}^{-M^2}| - |\Delta  \cap g_iF_{k_i}| - |g_i\Delta  \cap g_iF_{k_i}|  \nonumber   \\ 
 & \geq   \left(1-\frac{\varepsilon}{8} \right) |F_{k_i}| - \frac{\varepsilon}{16} |F_{k_i}| - \frac{\varepsilon}{16} |F_{k_i}| 
\nonumber \\ &\geq   (1-\frac{\varepsilon}{4}) |F_{k_i}| \nonumber
 \\
 &= (1-\frac{\varepsilon}{4}) |T_i|. 
\end{align}
 \par 
Let $R_i = E_i \setminus (\Delta \cup g_i \Delta) \subset g_i F_{k_i}^{-M^2} $ for $i \in I$. 
For $i \in I$, let $Z_i= \{0\}^{G \setminus E_iM} \times K^{E_iM}$.  Since $M$ is a memory set of $\sigma_s$ and $R_iM \subset E_iM$,  the NUCA $\sigma_s$ induces by restriction linear maps 
$\varphi_i \colon  Z_i \to K^{R_i}$  
for $i \in I$. 
By \eqref{e:linera-GOE-nuca-mdim-pre-1-c}, \eqref{e:linera-GOE-nuca-mdim-pre-1-f}, and \eqref{e:linera-GOE-nuca-mdim-pre-1-g}, and since $s\vert_{G\setminus \Delta}$ is constant, we have 
\begin{align} 
\label{e:linera-GOE-nuca-mdim-pre-1-image-est}
\dim \im \varphi_i  & = \dim \Gamma_{R_i} = \dim \left(\Gamma_{g_iF_{k_i}^{-M^2}}\right)_{R_i} \\
& = \dim \left(\Gamma_{F_{k_i}^{-M^2}}\right)_{g_i^{-1}R_i} \nonumber\\
& \geq  \dim \left(\Gamma_{F_{k_i}^{-M^2}}\right) 
- |F_{k_i}^{-M^2} \setminus g_i^{-1}R_i|\nonumber\\
& \geq  \left(\frac{\varepsilon}{2} + 1-d \right)|F_{k_i}| - (|F_{k_i}| - |g_i^{-1}R_i|) \nonumber\\
& \geq \left(\frac{\varepsilon}{4}+1-d\right) |F_{k_i}|.\nonumber
\end{align} 
\par 
Consequently, for every $i \in I$, we find that: 
\begin{align*}
    \dim \Ker \varphi_i  & = \dim Z_i - \dim \im \varphi_i \\
    & \leq |E_iM| - \left(\frac{\varepsilon}{4}+1-d\right) |F_{k_i}|\\
    & \leq |F_{k_i}| - \left(\frac{\varepsilon}{4}+1-d\right) |F_{k_i}| \\ & = \left(d- \frac{\varepsilon}{4}\right)|F_{k_i}| \\
    & = \left(d- \frac{\varepsilon}{4}\right)|T_{i}|. 
\end{align*}
and thus by linear algebra, we can  find a finite subset $S_i \subset E_iM$ with $|S_i|\leq \left(d- \frac{\varepsilon}{4}\right)|F_{k_i}|$ such that for $U_i=  \{0\}^{G\setminus E_iM} \times \{0\}^{S_i}  \times K^{E_iM \setminus S_i}$, the restriction  linear map 
$\varphi\vert_{U_i}$ is injective. 
Let us define 
\[
S=   G\,\setminus \bigcup_{i \in I}     \left(E_iM \setminus S_i\right). 
\]
\par 
It follows from the construction of the sets $S_i$ that $\sigma_s\vert_U$ is pre-injective where $U=\{0\}^S \times K^{G\setminus S}$. We claim that $\overline{d}_G(S) \leq d$. 
Indeed, for every $i \in I$, the relations  \eqref{e:linera-GOE-nuca-mdim-pre-1-b} and $|S_i|\leq \left(d- \frac{\varepsilon}{4}\right)|F_{k_i}|$ imply that: 
\begin{align*}
    |E_iM \setminus S_i|
 & \geq |E_iM| - |S_i| \geq |E_i| - |S_i| \\
 & \geq |F_{k_i}^{-M^2}| - |S_i|  = \left(1-\frac{\varepsilon}{8} \right) |F_{k_i}| - (d-\frac{\varepsilon}{4}) |F_{k_i}| \\
 & = \left( 1 - d + \frac{\varepsilon}{8} \right)|F_{k_i}|=\left( 1 - d + \frac{\varepsilon}{8} \right)|T_i|. 
\end{align*}
 \par 
 Therefore,  
$T'=\{(T_i, E_iM\setminus S_i)\}_{i \in I}$ is a $ \left(d - \frac{\varepsilon}{8} \right)$-disjoint and  $\left(1-{\frac{\varepsilon}{100}}\right)$-covering quasi-tiling of $G$ and we  infer from Lemma~\ref{l:technical-lem-density} that: 
\begin{align*}  
   \underline{D}_{G}\left( \bigcup_{i \in I}    E_iM \setminus S_i  \right) & \geq  \left(1 - d + \frac{\varepsilon}{8} \right) \left(1-\frac{\varepsilon}{100}\right) \nonumber\\
   & = 1-d +\varepsilon\left( \frac{1}{8} - \frac{1-d+\frac{\varepsilon}{8}}{100} \right) \nonumber\\
   & \geq 1 - d.  
\end{align*} 
\par 
Consequently, by Lemma~\ref{l:chain-densities}, we obtain: 
\begin{align*}
    \overline{D}_G(S)  = 1 -  \underline{D}_{G}\left( \bigcup_{i \in I}    E_iM \setminus S_i  \right) \leq d,  
\end{align*}
and the proof is thus complete. 
\end{proof}

By adopting the proof of Theorem~\ref{t:linera-GOE-nuca-mdim-pre}, we obtain the following result. 

\begin{theorem}
\label{t:characterization-linear-GOE-nuca-banach-dim}
Let $G$ be an infinite countable amenable group and let $K$ be a field. Then for every linear NUCA $\tau \colon K^G \to K^G$ with finite memory with $\underline{\mDim} (\im \tau ) > 1-d$, there exists $S\subset K^G$ with $\overline{D}_G(S)\leq d$  such that  $\tau\vert_U$ is pre-injective where $U=\{0\}^{S} \times K^{G \setminus S}$. 
\end{theorem}

\begin{proof}
The proof follows the same lines, \emph{mutatis mutandis}, of the proof of Theorem~\ref{t:linera-GOE-nuca-mdim-pre} where we set $\Delta = \varnothing$ so that $R_i=E_i$ and we simply replace  the relation  \eqref{e:linera-GOE-nuca-mdim-pre-1-c} by 
\begin{align}
\label{e:linera-GOE-nuca-mdim-pre-1-c-banach}
\frac{\dim \Gamma_{gF_n^{-M^2}}}{|F_n|} \geq \frac{\varepsilon}{2} + 1-d,  \quad \text{for all } g \in G.   
\end{align}
\par 
Consequently, we can modify the estimation \eqref{e:linera-GOE-nuca-mdim-pre-1-image-est} directly by: 
\begin{align*} 
\dim \im \varphi_i  & = \dim \Gamma_{R_i} = \dim \left(\Gamma_{g_iF_{k_i}^{-M^2}}\right)_{R_i} \\
& \geq  \dim \left(\Gamma_{g_iF_{k_i}^{-M^2}}\right) 
- |g_iF_{k_i}^{-M^2} \setminus  R_i|\\
& \geq  \left(\frac{\varepsilon}{2} + 1-d \right)|F_{k_i}| - (|F_{k_i}| - |R_i|) \\
& \geq \left(\frac{\varepsilon}{4}+1-d\right) |F_{k_i}|
\end{align*} 
and the rest of the proof is identical. 
\end{proof}

\par 
As a consequence, we obtain the following quantitative characterization of the partial pre-injectivity in terms of the mean dimension of the image of linear NUCA. 

\begin{theorem}
\label{t:characterization-linera-GOE-nuca-mdim-pre}
Let $G$ be an infinite countable amenable group and let $A$ be a finite-dimensional  vector space alphabet. For every linear NUCA  $\tau \colon A^G \to A^G$ with finite memory, the following hold:   
\begin{enumerate} [\rm (i)]
\item If $\tau\vert_U$ is pre-injective where $U=\{p\} \times A^{G \setminus S}$ and $p \in A^S$ for some subset $S \subset G$,  then  $\overline{\mdim} (\im \tau ) \geq (1-\overline{d}_G(S))\dim A$. 
    \item 
    If $\dim A=1$ and $\underline{\mDim} (\im \tau ) > 1-d$, then there exists $S\subset A^G$ with $\overline{D}_G(S)\leq d$  such that  $\tau\vert_U$ is pre-injective where $U=\{0\}^{S} \times A^{G \setminus S}$. 
\end{enumerate}
\end{theorem}

\begin{proof}
    Assertion (i) results,  \emph{mutatis mutandis}, from a similar proof of Theorem~\ref{t:GOE-nuca-c-1} where we replace $\overline{\ent}_\FF$ by $\overline{\mdim}_\FF$ and $\log |A|$ by $\dim A$. Moreover, with the same notations as in the proof of Theorem~\ref{t:GOE-nuca-c-1}, we deduce from the inclusion  $ \Gamma_{F_nM} \subset \Gamma_{F_n} \times A^{\partial_M F_n}$ that 
     $\dim \Gamma_{F_nM}\leq \dim \Gamma_{F_n} + \dim  A^{\partial_M F_n}$. Other modifications in the proof are trivial. 
     Assertion (ii) is nothing but the content of Theorem~\ref{t:characterization-linear-GOE-nuca-banach-dim}. The proof is thus complete. 
\end{proof}


\bibliographystyle{siam}

\begin{thebibliography}{10}

   
\bibitem{bartholdi-kielak}
{\sc L.~Bartholdi}, 
{\em Amenability of groups is characterized by Myhill's Theorem. With an appendix by D. Kielak}, 
J. Eur. Math. Soc. vol. 21, Issue 10 (2019), pp. 3191--3197. 

\bibitem{burton}
{\sc R. Burton and J.E Steif}, 
{\em Non-uniqueness of measures of maximal entropy for subshifts of finite type}, Ergodic Theory Dyn. Syst. 14 (1994), pp.~213--35. 


\bibitem{csc-goe-linear}
{\sc T.~Ceccherini-Silberstein and M.~Coornaert}, 
{\em The Garden of Eden theorem for linear cellular automata}, Ergod. Th. \& Dynam. Sys. 26 (2006), no. 1, pp.~53--68.

  
\bibitem{csc-sofic-linear}
\leavevmode\vrule height 2pt depth -1.6pt width 23pt, 
{\em Injective linear cellular automata and sofic groups}, 
Israel J. Math. 161 (2007), pp.~1--15.  


  
\bibitem{hedlund-csc}
\leavevmode\vrule height 2pt depth -1.6pt width 23pt, 
{\em A generalization of the Curtis-Hedlund theorem}, Theoret. Comput. Sci., 400 (2008), pp. 225–229

 
 

\bibitem{csc-myhill-strongly-irreducible}
\leavevmode\vrule height 2pt depth -1.6pt width 23pt,
{\em The Myhill property for strongly irreducible subshifts over amenable groups},  Monatsh. Math. 165 (2012), no. 2, pp.~155--172.

\bibitem{csc-nonlinearity}
\leavevmode\vrule height 2pt depth -1.6pt width 23pt,
{\em On the density of periodic configurations in strongly irreducible subshifts}, Nonlinearity 25 (2012),  pp.~2119--2131. doi:10.1088/0951-7715/25/7/2119

 
 



\bibitem{cscp-alg-goe}
{\sc T.~Ceccherini-Silberstein, M.~Coornaert, and X.~K. Phung},
{\em On the Garden of Eden theorem for endomorphisms of symbolic algebraic varieties}, 
 Pacific J. Math. 306 (2020), no. 1, pp~31--66. 
 

\bibitem{cscp-jpaa}
\leavevmode\vrule height 2pt depth -1.6pt width 23pt,
{\em On linear shifts of finite type and their endomorphisms},  J. Pure App. Algebra 226 (2022), no. 6. 


\bibitem{ceccherini}
{\sc T.~Ceccherini-Silberstein, A.~Mach{\`{\i}}, and F.~Scarabotti}, {\em
  Amenable groups and cellular automata}, Ann. Inst. Fourier (Grenoble), 49
  (1999), pp.~673--685.



\bibitem{Den-12a}  
{\sc A. Dennunzio, E. Formenti, and J. Provillard},  {\em Non-uniform cellular automata: Classes, dynamics, and decidability},  
Information and Computation
Volume 215, June 2012, Pages 32-46
 
\bibitem{Den-12b} 
{\sc A. Dennunzio, E. Formenti, and J. Provillard},   {\em Local rule distributions, language complexity and non-uniform
cellular automata}. Theoretical Computer Science, 504  (2013) 38–51. 



\bibitem{tiling-amenable}
{\sc T. Downarowicz, D. Huczek, and G. Zhang}
{\em Tilings of amenable groups}, Journal für die reine und angewandte Mathematik (Crelles Journal), vol. 2019, no. 747, 2019, pp.~277--298.  https://doi.org/10.1515/crelle-2016-0025

 


\bibitem{fiorenzi}
{\sc F. Fiorenzi}, 
{\em Cellular automata and strongly irreducible shifts of finite type}. 
Theoret. Comput. Sci. 
299, 477–493 (2003)

\bibitem{folner}
 {\sc E. Følner}, 
 {\em On groups with full Banach mean value}, 
 Math. Scand. 3 (1955), pp.~243--254.
 

 


\bibitem{gromov-esav}
{\sc M.~Gromov}, {\em Endomorphisms of symbolic algebraic varieties}, J. Eur.
  Math. Soc. (JEMS), 1 (1999), pp.~109--197.
  



\bibitem{gottschalk}
{\sc W.H.~Gottschalk}, 
{\em Some general dynamical notions}, 
Recent advances in topological dynamics, Springer, Berlin, 1973, pp. 120--125. Lecture Notes in Math. Vol. 318.


 
 
 
\bibitem{hedlund}  
{\sc G. A. Hedlund}, 
{\em Endomorphisms and automorphisms of the shift dynamical system}, 
Math. Systems Theory, 3 (1969), pp.~320--375.




\bibitem{kap}
{\sc I.~Kaplansky}, 
{\em Fields and Rings}, 
Chicago Lectures in Mathematics, University of Chicago Press, Chicago, 1969. 



\bibitem{lightwood} 
{\sc S.J. Lightwood},
{\em Morphisms from non-periodic $\Z^2$-subshifts: I. Constructing embeddings from
homomorphisms},
Ergod. Theory Dyn. Syst. 23 (2003), pp.~587--609. 

\bibitem{moore}
{\sc E.~F. Moore}, {\em Machine models of self-reproduction}, vol.~14 of Proc.
  Symp. Appl. Math., American Mathematical Society, Providence, 1963,
  pp.~17--34.


\bibitem{myhill}
{\sc J.~Myhill}, {\em The converse of {M}oore's {G}arden-of-{E}den theorem},
  Proc. Amer. Math. Soc., 14 (1963), pp.~685--686.


 


\bibitem{neumann-amenable}
{\sc J.~von Neumann}, 
{\em Zur allgemeinen Theorie des Masses}, Fund. Math. 13 (1929), pp.~73--116
and 333. = \emph{Collected works}, vol. I, pp~599--643.

\bibitem{neumann}
{\sc J.~von Neumann, In: Burks, A.W. (ed.)}, 
{\em  The Theory of Self-reproducing Automata}. University of Illinois Press, Urbana (1966). 



\bibitem{ornstein-weiss}
{\sc D. S. Ornstein and B. Weiss}, 
{\em Entropy and isomorphism theorems for
actions of amenable groups}, 
J. Analyse Math. 48, pp.~1--141 (1987). MR 910005 (88j:28014)



\bibitem{phung-2020}
{\sc X.K.~Phung}, 
{\em On sofic groups, Kaplansky's conjectures, and endomorphisms of pro-algebraic groups}, Journal of Algebra, 562 (2020), pp.~537--586. 

\bibitem{phung-shadowing} 
\leavevmode\vrule height 2pt depth -1.6pt width 23pt, 
{\em Shadowing for families of endomorphisms of generalized group shifts}, Discrete and Continuous Dynamical Systems 2022, 42 (1), pp.~285--299

 

\bibitem{phung-israel} 
\leavevmode\vrule height 2pt depth -1.6pt width 23pt, 
{\em On Dynamical Finiteness Properties of Algebraic Group Shifts}. Israel Journal of Mathematics (2022). https://doi.org/10.1007/s11856-022-2351-1

\bibitem{phung-post-surjective}
\leavevmode\vrule height 2pt depth -1.6pt width 23pt,
{\em On symbolic group varieties and dual surjunctivity}, to apear in Groups, Geometry, and Dynamics. arXiv:2111.02588



\bibitem{phung-tcs}
\leavevmode\vrule height 2pt depth -1.6pt width 23pt, 
{\em On invertible and stably reversible non-uniform cellular automata}, Theoretical Computer Science (2022). https://doi.org/10.1016/j.tcs.2022.09.011 

 

\bibitem{phung-geometric}
\leavevmode\vrule height 2pt depth -1.6pt width 23pt,
{\em A geometric generalization of Kaplansky's direct finiteness conjecture}, to appear in Proceedings of the American Mathematical Society.  arXiv:2111.07930 

\bibitem{phung-weakly}
\leavevmode\vrule height 2pt depth -1.6pt width 23pt, 
{\em Weakly surjunctive groups and symbolic group varieties}, preprint. arXiv:2111.13607


\bibitem{phung-twisted}
\leavevmode\vrule height 2pt depth -1.6pt width 23pt, 
{\em Stable finiteness of twisted group rings and noisy linear cellular automata}, preprint. arXiv:2209.06002


\bibitem{phung-decidable-nuca}
\leavevmode\vrule height 2pt depth -1.6pt width 23pt, 
{\em  Some pointwise and decidable properties of non-uniform cellular automata}, preprint. arXiv:2210.00676

\bibitem{phung-dual-linear-nuca}
\leavevmode\vrule height 2pt depth -1.6pt width 23pt, 
{\em On linear non-uniform cellular automata: duality and dynamics}, preprint.  arXiv:2208.13069


\bibitem{stan-amenable}
{\sc S. Wagon}, 
{\em The Banach-Tarski paradox}, Cambridge University Press, Cambridge, 1993. With a foreword by Jan Mycielski; Corrected reprint of the 1985 original.

\bibitem{ward}
{\sc T. Ward}, {\em Automorphisms of $\Z^d$-subshifts of finite type}, Indag. Math. (N.S.) 5 (1994), pp.~ 495--504
  

\end{thebibliography}

\end{document}